\documentclass[11pt]{amsart}

\usepackage{amssymb}
\usepackage{enumerate}
\usepackage{graphicx}
\usepackage{verbatim}

\newcommand{\llambda}{\underline{\lambda}}
\newcommand{\ulambda}{\overline{\lambda}}
\newcommand{\AAA}{\mathcal{A}}
\newcommand{\BBB}{\mathcal{B}}
\newcommand{\DDD}{\mathcal{D}}
\newcommand{\GGG}{\mathcal{G}}
\newcommand{\PPP}{\mathcal{P}}
\newcommand{\LLL}{\mathcal{L}}

\newcommand{\SSS}{\mathcal{S}}
\newcommand{\UUU}{\mathcal{U}}
\newcommand{\RR}{\mathbb{R}}
\newcommand{\UU}{\mathbf{U}}
\newcommand{\NN}{\mathbb{N}}
\newcommand{\CC}{\mathbb{C}}

\newcommand{\ph}{\varphi}
\newcommand{\eps}{\varepsilon}
\newcommand{\llim}{\varliminf}
\newcommand{\ulim}{\varlimsup}
\newcommand{\htop}{h_\mathrm{top}\,}
\newcommand{\abs}[1]{\left\lvert#1\right\rvert}

\DeclareMathOperator{\diam}{diam}
\DeclareMathOperator{\Crit}{Crit}

\theoremstyle{plain}
\newtheorem{theorem}{Theorem}[section]
\newtheorem{proposition}[theorem]{Proposition}
\newtheorem{corollary}[theorem]{Corollary}
\newtheorem{lemma}[theorem]{Lemma}

\theoremstyle{definition}
\newtheorem{definition}[theorem]{Definition}

\theoremstyle{remark}
\newtheorem{example}[theorem]{Example}
\newtheorem*{remark}{Remark}

\numberwithin{equation}{section}

\begin{document}
\title{Bowen's equation in the non-uniform setting}

\begin{abstract}
We show that Bowen's equation, which characterises the Hausdorff dimension of certain sets in terms of the topological pressure of an expanding conformal map, applies in greater generality than has been heretofore established.  In particular, we consider an arbitrary subset $Z$ of a compact metric space and require only that the lower Lyapunov exponents be positive on $Z$, together with a tempered contraction condition.  Among other things, this allows us to compute the dimension spectrum for Lyapunov exponents for maps with parabolic periodic points, and to relate the Hausdorff dimension to the topological entropy for arbitrary subsets of symbolic space with the appropriate metric.
\end{abstract}


\author{Vaughn Climenhaga}
\date{\today}
\thanks{This work is partially supported by NSF grant 0754911.}
\address{Department of Mathematics \\ McAllister Building \\ Pennsylvania State University \\ University Park, PA 16802, USA.}
\urladdr{http://www.math.psu.edu/climenha/}
\email{climenha@math.psu.edu}
\maketitle

\section{Introduction}

The first connection between topological pressure and Hausdorff dimension was given by Bowen~\cite{rB79}, who showed that for certain compact sets (quasi-circles) $J\subset \CC$ which arise as invariant sets of fractional linear transformations $f$ of the Riemann sphere, the Hausdorff dimension $t=\dim_H J$ is the unique root of the equation
\begin{equation}\label{eqn:Bowen}
P_J(-t\ph) = 0,
\end{equation}
where $P_J$ is the topological pressure of the map $f\colon J\to J$, and $\ph$ is the geometric potential $\ph(z) = \log |f'(z)|$.  Later, Ruelle showed that Bowen's equation~\eqref{eqn:Bowen} gives the Hausdorff dimension of $J$ whenever $f$ is a $C^{1+\eps}$ conformal map on a Riemannian manifold and $J$ is a repeller.  More precisely, he proved the following~\cite[Proposition 4]{dR82}:
 
\begin{theorem}\label{thm:ruelle}
Let $M$ be a Riemannian manifold and $V\subset M$ be open, and let $f\colon V\to M$ be $C^{1+\eps}$ and conformal (that is, $Df(x)$ is a scalar multiple of an isometry for every $x\in V$).  Suppose $J\subset V$ is a repeller---that is, it has the following properties:
\begin{enumerate}[(1)]
\item \emph{$J$ is compact.}
\item \emph{$J$ is maximal:}  $J = \{ x\in V \mid f^n(x) \in V \text{ for all } n>0 \}$.
\item \emph{$f$ is topologically mixing on $J$:}  For every open set $U\subset V$ such that $U\cap J \neq \emptyset$, there exists $n$ such that $f^n(U) \supset J$.
\item \emph{$f$ is uniformly expanding on $J$:}  There exist $C>0$ and $r>1$ such that $\| Df^n v \| \geq C r^n \| v \|$ for every tangent vector $v\in T_x M$ and every $n\geq 1$.
\end{enumerate}
Let $\ph(x) = \log \|Df(x)\|$.  Then Bowen's equation~\eqref{eqn:Bowen} has a unique root, and this root is equal to the Hausdorff dimension of $J$.
\end{theorem}

This result was eventually extended to the case where $f$ is $C^1$ by Gatzouras and Peres~\cite{GP97}.  One can also give a definition of conformal map in the case where $X$ is a metric space (not necessarily a manifold), and the analogue of Theorem~\ref{thm:ruelle} in this setting was proved by Rugh~\cite{hhR08}.

In all of these settings, one of the essential tools is the availability of geometric bounds that relate statically defined metric balls $B(x,r)$ (used in the definition of dimension) to dynamically defined Bowen balls $B(x,n,\delta)$ (used in the definition of pressure).  However, the above proofs differ in how these bounds are used.  The proofs given by Bowen, Ruelle, and Gatzouras and Peres all rely on the construction of a measure of full dimension (in particular, a measure that is equivalent to Hausdorff measure), which in turn relies on the aforementioned geometric bounds (among other things).  Rugh's proof, on the other hand, does not use measures and instead applies these bounds directly to the definitions of dimension and pressure.

These two methods of proof represent different approaches to the problem of using Bowen's equation to find the Hausdorff dimension of dynamically significant sets.  In this paper, we will follow the second approach (Rugh's) and avoid the use of measures; this will allow us to establish the analogue of Theorem~\ref{thm:ruelle} for a broad class of subsets of a repeller on which we may not have uniform expansion, and which need not carry any invariant measures.  First, however, we will mention some of the other settings in which the approach using measures of full dimension has been successful.

Working with maps in one real dimension, Urba\'nski~\cite{mU96} proved that the smallest root of~\eqref{eqn:Bowen} gives the Hausdorff dimension of a repeller $J$ that is expanding except on some set of indifferent fixed points, by finding a conformal measure that is the measure of full dimension.  Similar results for Julia sets of maps in one complex dimension were proved in~\cite{DU91, mU91}.  In fact, Bowen's equation is also known to give the Hausdorff dimension of the Julia set for a broad class of rational maps (those satisfying the \emph{topological Collet-Eckmann} condition) whose Julia sets even contain critical points~\cite{PRS03,PRS04}.  There are also situations where conformal measures can be built when $J$ is a non-compact set; for instance, when $J$ is the radial Julia set of a meromorphic function satisfying certain conditions~\cite{UZ04, MU08, MU10}.

Given a map $f$, all of the above results give the Hausdorff dimension of one very particular dynamically significant set $J$ via Bowen's equation.  It is natural to ask if one can find the Hausdorff dimension of subsets $Z\subset J$ via a similar approach.

For certain subsets, results in this direction are given by the multifractal analysis.  In the uniformly expanding case, the multifractal results in~\cite{BPS97,PW97,hW99} all boil down to the following result.  If $J$ is a conformal repeller and $\ph\colon J\to\RR$ is any H\"older continuous function, then for the one-parameter family of sets $K_\alpha \subset J$ given by
\[
K_\alpha = \left\{ x\in J \,\Big|\, \lim_{n\to\infty} \frac{S_n \ph(x)}{S_n \log \|Df(x)\|} = \alpha \right\},
\]
we may define a convex analytic function $T\colon \RR\to \RR$ implicitly by
\begin{equation}\label{eqn:implicitT}
P_J(q\ph - T(q)\log \|Df\|) = 0,
\end{equation}
and obtain $\dim_H K_\alpha$ as the Legendre transform of $T$:
\begin{equation}\label{eqn:dimHKa}
\dim_H K_\alpha = \inf_{q\in \RR} (T(q) - q\alpha).
\end{equation}
In the case $\ph\equiv 0$, $\alpha=0$, this reduces to Bowen's equation; for other values of $\ph$ and $\alpha$, this may be seen as a sort of (indirect) generalisation of Theorem~\ref{thm:ruelle}.  Analogous results for certain almost-expanding conformal maps with neutral fixed points are at the heart of the multifractal analyses in~\cite{PW99, kN00, GR09, GPR09, MU10}.

Once again, these results all rely on the construction of measures of full dimension on the sets $K_\alpha$ as Gibbs states $\nu_q$ for the family of potentials $q\ph - T(q)\log \|Df\|$, and so they do not generalise to more arbitrary subsets $Z \subset J$ (which may not support any invariant measures).  A more natural generalisation of Theorem~\ref{thm:ruelle} would be to obtain $\dim_H Z$ as the root of $P_Z(-t\log \|Df\|) = 0$ for some appropriate definition of $P_Z$ as the topological pressure \emph{on the set $Z$}, rather than on the entire repeller $J$.  (If such a generalisation is available, then the above multifractal results can be proved in a more general setting, where the measures $\nu_q$ are only required to be equilibrium states, and not necessarily Gibbs~\cite{vC10}.)

The appropriate definition of $P_Z$ was given by Pesin and Pitskel'~\cite{PP84}, characterising topological pressure as a Carath\'eodory dimension characteristic, and making sense of the expression $P_Z(\ph)$ for any subset $Z\subset J$.  (This extended the earlier definition by Bowen of topological entropy for non-compact sets~\cite{rB73}.)  Using the general theory of Carath\'eodory dimension characteristics introduced in~\cite{yP98}, Barreira and Schmeling~\cite{BS00} introduced the notion of the \emph{$u$-dimension} $\dim_u Z$ for positive functions $u$, showing that $\dim_u Z$ is the unique number $t$ such that $P_Z(-tu)=0$.  They also showed that for a subset $Z$ of a conformal repeller $J$, where we may take $u=\log \|Df\|>0$, we have $\dim_u Z=\dim_H Z$, and hence upon replacing $P_J$ with $P_Z$, the Hausdorff dimension of any subset $Z\subset J$ is given by Bowen's equation, whether or not $Z$ is compact or invariant.

Thus it has already been shown that in the uniformly expanding case, Theorem~\ref{thm:ruelle} holds not just for $J$ itself, but for any subset $Z\subset J$.  Furthermore, the aforementioned works of Urba\'nski \emph{et al} show that when we consider $J$ itself, there are many cases in which the requirement that $f$ be uniformly expanding can be replaced with rather weaker expansion properties.  However, there do not appear to be any results at present that combine these two directions, and give a Bowen's equation result for arbitrary sets $Z$ under properties weaker than uniform expansion (the closest results to this appear to be the multifractal results mentioned above).  Such a result is the purpose of this paper:  we show that the applicability of Bowen's equation to arbitrary $Z$ extends beyond the uniformly expanding case.

Indeed, given a conformal map $f$ without critical points or singularities, the only requirement we place on the expansion properties of $f$ is that every point $x$ of $Z$ has positive lower Lyapunov exponent, and that there not be too much contraction along the orbit of $x$ (see~\eqref{eqn:tempered} below---this is automatically satisfied if the Lyapunov exponent of $x$ exists or if $f$ is nowhere contracting).  We do not require any uniformity in these hypotheses; $Z$ may contain points with arbitrarily small or large Lyapunov exponents.  Furthermore, these hypotheses are only required to hold at points in $Z$, and not for other points in phase space.

This result has an immediate application to the multifractal formalism; we show that for any conformal map without critical points or singularities (no expansion properties are required), it allows us to compute the dimension spectrum for Lyapunov exponents directly from the entropy spectrum for Lyapunov exponents, which can in turn be obtained from the pressure function, provided the latter has nice properties.  Furthermore, this result is used in~\cite{vC10} to compute the dimension spectrum for pointwise dimensions given certain thermodynamic information.  Hopefully, this will eliminate some of the need for case-by-case analysis of various systems, and allow for more standardised techniques.

\section{Definitions and statement of result}

We consider a continuous map $f$ acting on a compact metric space $X$.

\begin{definition}
We say that $f\colon X\to X$ is \emph{conformal} with factor $a(x)$ if for every $x\in X$ we have
\begin{equation}\label{eqn:conformal}
a(x) = \lim_{y\to x} \frac{d(f(x),f(y))}{d(x,y)},
\end{equation}
where  $a\colon X\to [0,\infty)$ is continuous.  We denote the Birkhoff sums of $\log a$ by
\[
\lambda_n(x) = \frac 1n S_n (\log a)(x) = \frac 1n \sum_{k=0}^{n-1} \log a(f^k(x));
\]
the lower and upper limits of this sequence are the \emph{lower Lyapunov exponent} and \emph{upper Lyapunov exponent}, respectively:
\[
\llambda(x) = \llim_{n\to\infty} \lambda_n(x), \qquad
\ulambda(x) = \ulim_{n\to\infty} \lambda_n(x).
\]
If the two agree (that is, if the limit exists), then their common value is the \emph{Lyapunov exponent}:
\[
\lambda(x) = \lim_{n\to\infty} \lambda_n(x).
\]
Because $a$ is assumed to be continuous on a compact space $X$, it is bounded above (and hence $\ulambda(x)$ is as well); we do not allow maps with singularities.
\end{definition}

For later reference and uniformity of notation, we recall several equivalent definitions of Hausdorff dimension.

\begin{definition}\label{def:dimh}
Given $Z\subset X$ and $\eps>0$, let $\DDD(Z,\eps)$ denote the collection of countable open covers $\{ U_i \}_{i=1}^\infty$ of $Z$ for which $\diam U_i\leq \eps$ for all $i$.  For each $s\geq 0$, consider the set functions
\begin{align*}
m_H(Z,s,\eps) &= \inf_{\DDD(Z,\eps)} \sum_{U_i} (\diam U_i)^s, \\
m_H(Z,s) &= \lim_{\eps\to 0} m_H(Z,s,\eps).
\end{align*}
The \emph{Hausdorff dimension} of $Z$ is
\[
\dim_H Z = \inf \{ s>0 \mid m_H(Z,s)=0 \} = \sup \{ s>0 \mid m_H(Z,s)=\infty \}.
\]
It is straightforward to show that $m_H(Z,s)=\infty$ for all $s<\dim_H Z$, and that $m_H(Z,s)=0$ for all $s>\dim_H Z$.

One may equivalently define Hausdorff dimension using covers by open balls rather than arbitrary open sets; let $\DDD^b(Z,\eps)$ denote the collection of countable sets $\{(x_i,r_i)\}\subset Z\times (0,\eps]$ such that $Z\subset \bigcup_i B(x_i,r_i)$, and then define $m_H^b$ by
\begin{equation}\label{eqn:mHb}
m_H^b(Z,s,\eps) = \inf_{\DDD^b(Z,\eps)} \sum_i (\diam B(x_i,r_i))^s.
\end{equation}
Finally, define $m_H^b(Z,s)$ and $\dim_H^b Z$ by the same procedure as above; then  Proposition~\ref{prop:dimsequal} shows that $\dim_H^b Z = \dim_H Z$, so we are free to use either definition.

It is natural to replace~\eqref{eqn:mHb} with
\begin{equation}\label{eqn:mHb'}
m_H^{b'}(Z,s,\eps) = \inf_{\DDD^b(Z,\eps)} \sum_i (2r_i)^s;
\end{equation}
however, the two quantities are not necessarily equal, as we may have $\diam B(x,r) < 2r$ (if $x$ is an isolated point, for example, or if $X$ is homeomorphic to a Cantor set).  Nevertheless, Proposition~\ref{prop:dimsequal} shows that the resulting critical value $\dim_H^{b'} Z$ is equal to $\dim_H Z$.  (This result is straightforward, but does not appear in the standard references on Hausdorff dimension, so we include a proof below for completeness.)
\end{definition}

The following definition defines topological pressure for arbitrary sets (which are not necessarily compact or invariant) as a Carath\'eodory dimension characteristic; mirroring the definition of Hausdorff dimension.  

\begin{definition}
Let $X$ be a compact metric space and consider a map $f\colon X\to X$.  The Bowen ball of radius $\delta$ and order $n$ is
\[
B(x,n,\delta) = \{y\in X \mid d(f^k(y),f^k(x))<\delta \text{ for all } 0\leq k\leq n \}.
\]
Now fix a potential function $\ph\colon X\to \RR$.  Given $Z\subset X$, $\delta>0$, and $N\in\NN$, let $\PPP(Z,N,\delta)$ be the collection of countable sets $\{ (x_i,n_i) \} \subset Z\times \{N,N+1,\dots\}$ such that $Z \subset \bigcup_i B(x_i,n_i,\delta)$.  For each $s\in\RR$, consider the set functions
\begin{equation}\label{eqn:mZa}
\begin{aligned}
m_P(Z,s,\ph,N,\delta)&= \inf_{\PPP(Z,N,\delta)} \sum_{(x_i,n_i)} 
\exp\left(-n_i s + S_{n_i} \ph(x_i) \right), \\
m_P(Z,s,\ph,\delta)&=\lim_{N\to\infty} m_P(Z,s,\ph,N,\delta).
\end{aligned}
\end{equation}
This function is non-increasing in $s$, and takes values $\infty$ and $0$ at all but at most one value of $s$.  Denoting the critical value of $s$ by
\begin{align*}
P_Z(\ph,\delta) &= \inf \{s\in\RR \mid m_P(Z,s,\ph,\delta)=0\} \\
&= \sup \{s\in\RR \mid m_P(Z,s,\ph,\delta)=\infty\},
\end{align*}
we get $m_P(Z,s,\ph,\delta)=\infty$ when $s<P_Z(\ph,\delta)$, and $0$ when $s>P_Z(\ph,\delta)$.

The \emph{topological pressure} of $\ph$ on $Z$ is $P_Z(\ph) = \lim_{\delta\to 0} P_Z(\ph,\delta)$; the limit exists because given $\delta_1 < \delta_2$, we have $\PPP(Z,N,\delta_1) \subset \PPP(Z,N,\delta_2)$, and hence $m_P(Z,s,\ph,\delta_1) \geq m_P(Z,s,\ph,\delta_2)$, so $P_Z(\ph,\delta_1) \geq P_Z(\ph,\delta_2)$.  

In the particular case $\ph=0$, we get the topological entropy $\htop(Z) = P_Z(0) \geq 0$, which exactly mirrors the second definition of Hausdorff dimension, replacing the balls $B(x_i, r_i)$ with Bowen balls $B(x_i,n_i,\delta)$.
\end{definition}

\begin{remark}
We show below (Proposition~\ref{prop:defsequal}) that if $f$ and $\ph$ are continuous (which is the case in this paper), then this definition is equivalent to the definition given by Pesin and Pitskel'~\cite{PP84} (see also~\cite{yP98}).  In particular, when $f$ and $\ph$ are continuous, and $Z$ is compact and invariant, this gives us an alternate method to compute the classical topological pressure.  

The definition given here is easier to use for our present purposes, as we will consider the case where $\ph(x)$ is a multiple of $\log a(x)$, and so the sums $S_{n_i} \ph(x_i)$ which appear in the definition of pressure are proportional to the amount of expansion \emph{along the orbit of $x_i$}.  This fact allows us to relate the Bowen balls $B(x_i,n_i,\delta)$ centred at $x_i$ to the usual balls $B(x_i,r_i)$, for appropriate values of $r_i$, and hence to draw a connection between the definition of Hausdorff dimension (via balls) and the present definition of topological pressure.
\end{remark}

Our main result relates the Hausdorff dimension of $Z$ to the topological pressure of $\log a$ on $Z$, provided every point in $Z$ has positive lower Lyapunov exponent and satisfies the following \emph{tempered contraction} condition:
\begin{equation}\label{eqn:tempered}
\inf_{\substack{n\in \NN \\ 0\leq k\leq n}} \{ S_{n-k} \log a(f^k(x)) + n\eps \} > -\infty \text{ for every $\eps>0$.}
\end{equation}
Denote by $\BBB$ the set of all points in $X$ which satisfy~\eqref{eqn:tempered}.

Observe that~\eqref{eqn:tempered} is automatically satisfied if $a(x)\geq 1$ for all $x\in X$, and so in this case $\BBB=X$.  Another case in which $x$ satisfies~\eqref{eqn:tempered} is when $x$ has \emph{bounded contraction}:  $\inf \{S_{n-k} \log a(f^k(x))\mid n\in \NN, 0\leq k\leq n \} > -\infty$.

Proposition~\ref{prop:Lyap-tempered} shows that if the Lyapunov exponent of $x$ exists---that is, if $\llambda(x) = \ulambda(x)$---then $x$ satisfies~\eqref{eqn:tempered}.

Given $E\subset\RR$, we denote by $\AAA(E)$ the set of points along whose orbits all the asymptotic exponential expansion rates of the map $f$ lie in $E$:
\[
\AAA(E) = \{ x\in X \mid [\llambda(x),\ulambda(x)]\subset E\}.
\]
In particular, $\AAA((0,\infty))$ is the set of all points for which $\llambda(x)>0$.  Our main result deals with subsets $Z\subset X$ that lie in both $\AAA((0,\infty))$ and $\BBB$.  (Observe that by Proposition~\ref{prop:Lyap-tempered}, $\AAA(\alpha) = \AAA(\{\alpha\}) \subset \BBB$ for every $\alpha>0$.)

\begin{theorem}\label{thm:main}
Let $X$ be a compact metric space and $f\colon X\to X$ be continuous and conformal with factor $a(x)$.  Suppose that $f$ has no critical points and no singularities---that is, that $0<a(x)<\infty$ for all $x\in X$.  Consider $Z\subset \AAA((0,\infty)) \cap \BBB$.  Then the Hausdorff dimension of $Z$ is given by
\begin{equation}\label{eqn:dimhz}
\begin{aligned}
\dim_H Z = t^* &= \sup\{t\geq 0 \mid P_Z(-t\log a) > 0 \} \\
&= \inf\{t\geq 0 \mid P_Z(-t\log a) \leq 0 \}.
\end{aligned}
\end{equation}
Furthermore, if $Z\subset \AAA((\alpha,\infty))\cap \BBB$ for some $\alpha>0$ (that is, the lower Lyapunov exponents of points in $Z$ are \emph{uniformly} positive), then $t^*$ is the unique root of Bowen's equation
\begin{equation}\label{eqn:Bowen2}
P_Z(-t\log a) = 0.
\end{equation}
Finally, if $Z\subset \AAA(\alpha)$ for some $\alpha>0$, then $P_Z(-t\log a) = \htop Z - t\alpha$, and hence
\begin{equation}\label{eqn:Bowen3}
\dim_H Z = \frac 1\alpha \htop Z.
\end{equation}
\end{theorem}

Before proceeding to specific examples and to the proofs, we make a few remarks on Theorem~\ref{thm:main} in some standard settings.
\begin{enumerate}
\item  For expanding conformal maps ($a(x)>1$ for all $x$), we have $\BBB = X$, and Theorem~\ref{thm:main} reduces to Barreira and Schmeling's generalisation of Theorem~\ref{thm:ruelle}, although we work in the slightly more general setting where $X$ need not be a manifold.
\item  For almost expanding conformal maps (maps which are expanding away from a collection of indifferent periodic points), we have $a(x)\geq 1$ for all $x$, and so $\BBB = X$; thus the theorem applies, showing that Bowen's formula gives the Hausdorff dimension of any set which does not contain any points with zero lower Lyapunov exponent.  This complements the results in~\cite{DU91, mU91, mU96}, which give the Hausdorff dimension of the \emph{entire} Julia set for a large family of almost expanding conformal maps, but have nothing to say about arbitrary subsets of the Julia set.  (Observe that because the Julia set contains points with zero Lyapunov exponent, Theorem~\ref{thm:main} does not give the Hausdorff dimension of the entire Julia set.)
\item  For maps with some contracting regions ($a(x)<1$) but no critical points ($a(x)=0$), we cannot rule out the possibility that $\BBB \neq X$.  However, the result still holds for $Z\subset X$ as long as every point $x\in Z$ satisfies~\eqref{eqn:tempered} and has positive lower Lyapunov exponent.  In particular, if the Lyapunov exponent is constant and positive on $Z$, then~\eqref{eqn:Bowen3} relates the Hausdorff dimension of $Z$ to the topological entropy of $Z$.
\end{enumerate}

\section{An application to the multifractal formalism}\label{sec:multifractal}

The multifractal formalism characterises dynamical systems in terms of various multifractal spectra, of which an overview may be found in~\cite{BPS97}.  The present result gives a general relationship between two of these spectra, which are both defined in terms of the level sets of Lyapunov exponents of a conformal map:
\[
\AAA(\alpha) = \AAA(\{\alpha\}) = \{ x\in X \mid \lambda(x) = \alpha \}.
\]
The \emph{dimension spectrum for Lyapunov exponents} of $f$ is
\[
\LLL_D(\alpha) = \dim_H \AAA(\alpha),
\]
and the \emph{entropy spectrum for Lyapunov exponents} of $f$ is
\[
\LLL_E(\alpha) = \htop \AAA(\alpha).
\]
These spectra have been studied for conformal repellers by Weiss~\cite{hW99}; in the non-uniform setting, they have been studied for Manneville--Pomeau maps (that is, one-dimensional Markov maps with a neutral fixed point) by Pollicott and Weiss~\cite{PW99}, Nakaishi~\cite{kN00}, and Gelfert and Rams~\cite{GR09}, and for rational maps by Gelfert, Przytycki, and Rams~\cite{GPR09}.

Of the two, $\LLL_E$ is \emph{a priori} the easier to investigate, as it can in many cases be obtained as the Legendre transform of the function
\begin{equation}\label{eqn:T}
T \colon t\mapsto P_X(-t\log a).
\end{equation}
Indeed, the following theorem is proved in~\cite{vC10}:
 
\begin{theorem}\label{thm:blackbox}
Let $f\colon X\to X$ be conformal with factor $a(x)$, and suppose that $f$ has no critical points or singularities, so that $0<a(x)<\infty$ for all $x\in X$.  Let $t_1,t_2 \in [-\infty, \infty]$ be such that the following hold for every $t\in (t_1,t_2)$:
\begin{enumerate}
\item An equilibrium state exists for the potential function $-t\log a$ (this is true for all $t$ if $f$ is expansive);
\item The function $T$ given in~\eqref{eqn:T} is differentiable at $t$.
\end{enumerate}
Let $\alpha_1 = -\lim_{t\to t_2^-} T'(t)$ and $\alpha_2 = -\lim_{t\to t_1^+} T'(t)$.  Then $\alpha_1 < \alpha_2$, and for all $\alpha \in (\alpha_1, \alpha_2)$, the entropy spectrum for Lyapunov exponents is given by
\begin{equation}\label{eqn:lyapspec}
\LLL_E(\alpha) = \inf_{t\in \RR} ( T(t) - t\alpha ).
\end{equation}
If in addition $T$ is strictly convex on $(t_1,t_2)$, then $\LLL_E$ is differentiable on $(\alpha_1,\alpha_2)$, and has the same regularity as $T$.
\end{theorem}

Once we know $\LLL_E(\alpha)$ for a given $\alpha>0$, we may apply~\eqref{eqn:Bowen3} to the level set $\AAA(\alpha)$, obtaining
\[
\dim_H \AAA(\alpha) = \frac{1}{\alpha} \htop \AAA(\alpha).
\]
Thus the dimension spectrum for Lyapunov exponents is determined by the entropy spectrum for Lyapunov exponents as follows:
\begin{equation}\label{eqn:LDviaLE}
\LLL_D(\alpha) = \frac{1}{\alpha} \LLL_E(\alpha),
\end{equation}
for all $0 < \alpha < \infty$; in conjunction with Theorem~\ref{thm:blackbox}, this establishes the multifractal formalism for both spectra when $T$ is differentiable.

\begin{example}
In the setting of Theorem~\ref{thm:ruelle}, where $f$ is a $C^{1+\eps}$ expanding conformal map on a repeller $J$, it is well-known that the pressure function $T \colon t \mapsto P_J(-t\log a)$ is real analytic and strictly convex (provided $\log a$ is not cohomologous to a constant, or equivalently, that the measure of maximal dimension and the measure of maximal entropy do not coincide).  It is shown in~\cite{hW99} that in this case the dimension spectrum for Lyapunov exponents is real analytic on an interval $(\alpha_1, \alpha_2)$, and may be obtained in terms of the Legendre transform of the pressure function.\footnote{Weiss also claims that the spectrum is concave, but Iommi and Kiwi have shown that there are examples in which this is not the case~\cite{IK09}.}

The proof in~\cite{hW99} is roundabout, and analyses $\LLL_D(\alpha)$ in terms of the dimension spectrum for pointwise dimensions of a measure of maximal entropy, by showing that for such a measure the level sets of the pointwise dimension coincide with the level sets of the Lyapunov exponent (this may also be shown using the fact that the local entropy of such a measure is constant everywhere and applying Lemma~\ref{lem:well-behaved} below), and then applying results from~\cite{PW97}.

In contrast, the proof of Theorem~\ref{thm:blackbox} does not involve any other spectra, and together with~\eqref{eqn:LDviaLE}, this gives a more direct proof of the following well-known result.
\begin{proposition}
Let $f\colon V\to M$ be as in Theorem~\ref{thm:ruelle}, and let $J$ be a uniformly expanding repeller.  Then the Lyapunov spectra of $f$ are given in terms of the Legendre transform of the pressure function as follows:
\begin{equation}\label{eqn:legendre}
\begin{aligned}
\LLL_E(\alpha) &= \inf_{t\in\RR} ( P_J(-t\log a) - \alpha t ), \\
\LLL_D(\alpha) &= \frac 1\alpha \inf_{t\in\RR} ( P_J(-t\log a) - \alpha t ).
\end{aligned}
\end{equation}
\end{proposition} 

In particular, if $\log a$ is not cohomologous to a constant, then the spectrum $\LLL_E$ is strictly concave, and both spectra are real analytic (this follows from analyticity of the pressure function and standard properties of the Legendre transform).
\end{example}

\begin{example}
Let $f\colon \overline{\CC} \to \overline{\CC}$ be a parabolic rational map of the Riemann sphere; that is, a rational map such that the Julia set $J$ contains at least one indifferent fixed point (that is, a fixed point $z_0$ for which $|f'(z_0)|=1$), but does not contain any critical points.  Then the map $f\colon J\to J$ satisfies the hypotheses of Theorem~\ref{thm:main}, and so~\eqref{eqn:LDviaLE} gives $\LLL_D$ in terms of $\LLL_E$.

\begin{figure}[tbp]
	\includegraphics{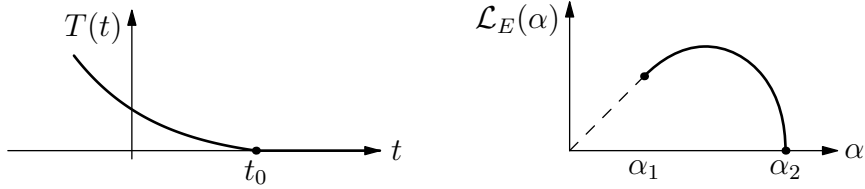}
	\caption{The pressure function and Lyapunov spectrum for a parabolic rational map.}
	\label{fig:graphs}
\end{figure}

Following Makarov and Smirnov~\cite{MS00}, we say that $f$ is \emph{exceptional} if there is a finite, non-empty set $\Sigma \subset \overline{\CC}$ such that $f^{-1}(\Sigma) \setminus \Crit f = \Sigma$, where $\Crit f$ is the set of critical points of $f$.  Combining the results in~\cite{MS00} with~\cite[Corollary D.1 and Theorem G]{hH08}, we see that if $f$ is non-exceptional, then the graph of the function $T$ is as shown in Figure~\ref{fig:graphs}.  In particular, $T$ is analytic and strictly convex on $(-\infty, t_0)$, where $t_0 = \dim_H J(f)$, and so writing
\[
\alpha_1 = -\lim_{t \to t_0^-} T'(t), \qquad \alpha_2 = -\lim_{t \to -\infty} T'(t),
\]
it follows from Theorem~\ref{thm:blackbox} and~\eqref{eqn:LDviaLE} that the Lyapunov spectra $\LLL_E$ and $\LLL_D$ are given by~\eqref{eqn:legendre} on $(\alpha_1, \alpha_2)$.

This result is obtained by other methods in~\cite{GPR09}, where it is also shown that the spectrum $\LLL_E$ is linear on $[0,\alpha_1]$ (the dotted line in Figure~\ref{fig:graphs}), with $\LLL_E(\alpha) = \alpha$.  Assuming this result,~\eqref{eqn:LDviaLE} shows that $\LLL_D(\alpha) =  1$ for $0<\alpha\leq\alpha_1$, a fact which must be proved separately in~\cite{GPR09}.
\end{example}

\section{An application to symbolic dynamics}

We now describe a class of systems to which these results may be applied, for which the phase space is not a manifold.  Fix an integer $k\geq 2$, and let $X = \Sigma_k^+$ be the full one-sided shift on $k$ symbols.  Given $x,y\in X$, let $x\wedge y$ denote the common prefix of $x$ and $y$---that is, if $n$ is the unique integer such that $x_i = y_i$ for all $1\leq i\leq n$, but $x_{n+1} \neq y_{n+1}$, then
\[
x\wedge y = x_1 \dots x_n = y_1 \dots y_n.
\]
Let $\psi\colon \bigcup_{n\geq 0} \{1,\dots, k\}^n \to \RR^+$ be a function defined on the space of all finite words on the alphabet $\{1,\dots,k\}$, and suppose that $\psi$ is such that for every $x\in X$, the sequence $\{\psi(x_1\dots x_n)\}$ is non-increasing and approaches $0$ as $n\to\infty$.  Then $d(x,y) = \psi(x\wedge y)$ defines a metric on $X=\Sigma_k^+$; to prove this, one needs only verify the triangle inequality, or equivalently, show that
\[
\psi(x\wedge z) \leq \psi(x\wedge y) + \psi(y\wedge z)
\]
for every $x,y,z\in X$.  This follows from the observation that if $n$ is the length of the common prefix of $x$ and $z$, then either $x_{n+1} \neq y_{n+1}$ or $y_{n+1} \neq z_{n+1}$: without loss of generality, suppose the first holds, and then we have
\[
\psi(x\wedge y) = \psi(x_1\dots x_m) \geq \psi(x_1\dots x_n) = \psi(x\wedge z)
\]
for some $0\leq m\leq n$.  Thus $d$ is a metric, and the requirement that $\psi(x_1\dots x_n) \to 0$ guarantees that $d$ induces the product topology on $X = \{1,\dots,k\}^\NN$.

In order for the shift $\sigma$ to be conformal, we require the following limit to exist for every $x\in X$:
\begin{equation}\label{eqn:axpsi}
a(x) = \lim_{n\to\infty} \frac{\psi(x_2\dots x_n)}{\psi(x_1\dots x_n)}.
\end{equation}
Furthermore, we demand that $a(x)$ depend continuously on $x$.  If these conditions are satisfied, then the shift $\sigma$ is conformal with factor $a(x)$ given by~\eqref{eqn:axpsi}:  indeed, given any $x,y \in X$ such that the length of the common prefix is $n$, we have
\[
\frac{d(\sigma(x),\sigma(y))}{d(x,y)} = 
\frac{\psi(\sigma(x) \wedge \sigma(y))}{\psi(x\wedge y)} =
\frac{\psi(x_2\dots x_n)}{\psi(x_1\dots x_n)},
\]
and since $y\to x$ if and only if $n\to \infty$, this gives
\[
\lim_{y\to x} \frac{d(\sigma(x),\sigma(y))}{d(x,y)} =
\lim_{n\to\infty} \frac{\psi(x_2\dots x_n)}{\psi(x_1\dots x_n)} = a(x).
\]

Given the above conditions, we may apply Theorem~\ref{thm:main} to subsets $Z\subset X$ on which the lower Lyapunov exponents are positive and we have tempered contraction.  We now describe some simple candidates for the function $\psi$, for which the results take on a straightforward form (and for which tempered contraction is automatic).

\begin{example}
Fix $\theta > 1$ and let $\psi(x_1\dots x_n) = \theta^{-n}$, so that $d(x,y) = \theta^{-n}$, where $n+1$ is the first entry in which $x$ and $y$ differ.  Then $a(x) = \theta$ for every $x\in X$, and it follows from~\eqref{eqn:Bowen3} that for every $Z\subset X$, we have
\[
\dim_H Z = \frac{\htop Z}{\log \theta}.
\]
\end{example}

\begin{example}
Fix $\theta_1, \theta_2, \dots, \theta_k \geq 1$ and define $\psi$ by
\[
\psi(x_1\dots x_n) = \frac 1n (\theta_{x_1} \theta_{x_2} \cdots \theta_{x_n})^{-1}.
\]
(The factor of $\frac 1n$ is necessary to ensure that $\psi(x_1\dots x_n) \to 0$ even if all but finitely many of the $\theta_{x_i}$ are equal to $1$; it can be omitted if $\theta_j > 1$ for all $j$.)  Then $a(x) = \theta_{x_1}$ is continuous, and $a(x)\geq 1$ for all $x$, so $\BBB = X$.

It follows that~\eqref{eqn:dimhz} gives the Hausdorff dimension of any set $Z\subset X$ on which the lower Lyapunov exponents are positive (in this case, the lower Lyapunov exponent of $x$ is determined solely by the asymptotic frequency of the various symbols $1,\dots,k$ in the expansion of $x$).

If we consider the level sets of Lyapunov exponents, then we see that the entropy spectrum for Lyapunov exponents and the dimension spectrum for Lyapunov exponents are once again related by~\eqref{eqn:LDviaLE}.
\end{example}

\section{Preparatory results}

We proceed now to the proofs, beginning with two propositions allowing us to use the definitions in a form that will make later computations more convenient.

\begin{proposition}\label{prop:dimsequal}
Let $X$ be a separable metric space, fix $Z\subset X$, and let $\dim_H Z$, $\dim_H^b Z$, and $\dim_H^{b'} Z$ be as in Definition~\ref{def:dimh}.  Then all three quantities are equal.
\end{proposition}
\begin{proof}
To see that $\dim_H Z = \dim_H^b Z$, it suffices to show that
\[
2^{-s} m_H^b(Z,s,\eps) \leq m_H(Z,s,\eps) \leq m_H^b(Z,s,\eps/2).
\]
The first inequality follows by associating to every $\eps$-cover $\{U_i\}\in \DDD(Z,\eps)$ the set $\{(x_i,r_i)\}\in \DDD^b(Z,\eps)$, where $x_i\in U_i$ is arbitrary and $r_i = \diam U_i$.  The second inequality follows by associating to every $\{(x_i,r_i)\} \in \DDD^b(Z,\eps/2)$ the $\eps$-cover $\{B(x_i,r_i)\}$.  (This half of the proposition may be found in any standard reference on Hausdorff dimension.)

To show that $\dim_H^b Z = \dim_H^{b'} Z$, we show that
\begin{equation}\label{eqn:mHbb'}
m_H^b(Z,s,\eps) \leq m_H^{b'}(Z,s,\eps) \leq 2^s m_H^b(Z,s,\eps).
\end{equation}
The first inequality is immediate since $\diam B(x,r) \leq 2r$ for every $x\in X$ and $r>0$.  For the second inequality, we observe that by separability, there are at most countably many isolated points in $X$, and that removing a countable number of isolated points does not affect the value of $m_H^{b'}(Z,s,\eps)$ or $m_H^b(Z,s,\eps)$; thus we may assume without loss of generality that $X$ has no isolated points.  Given $(x_i,r_i)$, let $t_i$ be given by
\[
t_i = \sup\{ t\in [0,r_i] \mid d(x_i,y)=t \text{ for some } y\in B(x_i,r_i) \};
\]
because $x_i$ is not isolated, we have $0<t_i \leq r_i$.  Furthermore, we have
\[
\diam B(x_i,t_i) \geq d(x_i,y) = t_i,
\]
and so
\[
\sum_i (\diam B(x_i,t_i))^s \geq \sum_i t_i^s = 2^{-s} \sum_i (2t_i)^s.
\]
Taking the infimum over all $\{(x_i,r_i)\}\in \DDD^b(Z,\eps)$ gives the second inequality in~\eqref{eqn:mHbb'}, and we are done.
\end{proof}

\begin{proposition}\label{prop:defsequal}
For continuous $f$ and $\ph$, the definition of pressure given here is equivalent to the definition given in~\cite{yP98}.
\end{proposition}
\begin{proof}
In~\cite{yP98}, Pesin defines topological pressure as follows.  Given a compact metric space $X$, a continuous map $f\colon X\to X$, and a continuous function $\ph\colon X\to \RR$, we fix a finite open cover $\UUU$ of $X$, and let $\SSS_m(\UUU)$ denote the set of all strings $\UU = \{U_{w_1} \dots U_{w_m} \mid U_{w_j} \in \UUU\}$ of length $m=m(\UU)$.  We write $\SSS = \SSS(\UUU) = \bigcup_{m\geq 0} \SSS_m(\UUU)$.

Now to each string $\UU \in \SSS(\UUU)$ we associate the set
\[
X(\UU) = \{x\in X \mid f^{j-1}(x) \in U_{w_j} \text{ for all } j=1,\dots, m(\UU) \};
\]
given $Z\subset X$ and $N\in\NN$, we let $\SSS(Z,\UUU,N)$ denote the set of all finite or countable collections $\GGG$ of strings of length at least $N$ which cover $Z$; that is, $\GGG\subset \SSS(\UUU)$ is in $\SSS(Z,\UUU,N)$ if and only if
\begin{enumerate}
\item $m(\UU)\geq N$ for all $\UU\in \GGG$, and also
\item $\bigcup_{\UU\in\GGG} X(\UU) \supset Z$.
\end{enumerate}
Then we define a set function by
\begin{multline}\label{eqn:mprime}
m_P'(Z,\ph,\UUU,s,N) \\ = \inf_{\SSS(Z,\UUU,N)} \left\{ 
\sum_{\UU\in\GGG} \exp \left(-s m(\UU) + \sup_{x\in X(\UU)} S_{m(\UU)} \ph(x) \right) \right \}
\end{multline}
and the critical value of $m_P'(Z,\ph,\UUU,s) = \lim_{N\to\infty} m_P'(Z,\ph,\UUU,s,N)$ by
\[
P'_Z(\ph,\UUU) = \inf\{s \mid m_P'(Z,\ph,\UUU,s) = 0 \} = \sup\{ s \mid m_P'(Z,\ph,\UUU,s)=\infty \}.
\]
(We write $m_P'$ and $P'$ to distinguish these from our definitions given earlier.)  The topological pressure is $P'_Z(\ph) = \lim_{\abs{\UUU}\to 0} P'_Z(\ph,\UUU)$, where $\abs{\UUU} = \max\{\diam U_i \mid U_i \in \UUU\}$ is the diameter of the cover $\UUU$.

Given $\delta>0$, let
\[
\eps(\delta) = \sup\{ |\ph(x) - \ph(y)| \mid d(x,y)<\delta \},
\]
and observe that since $\ph$ is continuous and $X$ is compact, $\ph$ is in fact uniformly continuous, hence $\eps(\delta)$ is finite, and  $\lim_{\delta\to 0} \eps(\delta)=0$.  Furthermore, given $x\in X$, $y\in B(x,n,\delta)$, we have
\[
|S_n\ph(x) - S_n\ph(y)| < n\eps(\delta).
\]

Now for a fixed $\delta>0$, we choose a cover $\UUU$ with $\abs{\UUU}<\eps(\delta)$.  Let $\gamma(\UUU)$ be the Lebesgue number of $\UUU$, and consider $\{(x_i,n_i)\} \in \PPP(Z,N,\gamma(\UUU))$.  Then for each $(x_i,n_i)$ there exists $\UU_i \in \SSS_{n_i}(\UUU)$ such that $B(x_i,n_i,\gamma(\UUU)) \subset X(\UU_i)$; let $\GGG' = \{ \UU_i \}$, and then
\begin{align*}
m_P'(Z,\ph,\UUU,s,N) &= \inf_{\SSS(Z,N,\delta)} \sum_{\UU\in\GGG} \exp\left(-s m(\UU) + \sup_{x\in X(\UU)} S_{m(\UU)} \ph(x)\right) \\
&\leq \sum_{\UU_i \in\GGG'} \exp\left(-s m(\UU_i ) + \sup_{x\in X(\UU_i)} S_{m(\UU_i )}\ph(x)\right) \\
&\leq \sum_{(x_i,n_i)} \exp\left(-n_i (s-\eps(\delta)) + S_{n_i}\ph(x_i) \right).
\end{align*}
Since the collection $\{(x_i,n_i)\}$ was arbitrary, we have
\[
m_P'(Z,\ph,\UUU,s,N) \leq m_P(Z,s-\eps(\delta),\ph,N,\gamma(\UUU)).
\]
Taking the limit $N\to\infty$ yields
\[
P'_Z(\ph,\UUU) \leq P_Z(\ph,\gamma(\UUU)) - \eps(\delta),
\]
and as $\delta\to 0$ we obtain
\[
P'_Z(\ph) \leq P_Z(\ph).
\]

For the other inequality, fix a cover $\UUU$ of $X$, with $\abs{\UUU}<\delta$.  Given $\GGG\in \SSS(Z,\UUU,N)$, we may assume without loss of generality that for every $\UU\in\GGG$, we have $X(\UU)\cap Z \neq \emptyset$ (otherwise we may eliminate some sets from $\GGG$, which does not increase the sum in~\eqref{eqn:mprime}).  Thus for each such $\UU$, we choose $x_\UU \in X(\UU) \cap Z$; we see that $X(\UU) \subset B(x_\UU,m(\UU),\delta)$, and so
\begin{align*}
m_P'(Z,\ph,\UUU,s,N) &= \inf_{\SSS(Z,\UUU,N)} \sum_{\UU\in\GGG} \exp\left(-s m(\UU) + \sup_{x\in X(\UU)} S_{m(\UU)} \ph(x)\right) \\
&\geq \inf_{\PPP(Z,N,\delta)} \sum_{(x_i,n_i)} \exp\left(-n_i s + S_{n_i}\ph(x_i)\right) \\
&= m_P(Z,s,\ph,N,\delta).
\end{align*}
Thus $P'_Z(\ph,\UUU)\geq P_Z(\ph,\delta)$, and taking the limit as $\delta\to 0$ gives
\[
P'_Z(\ph) \geq P_Z(\ph),
\]
which completes the proof.
\end{proof}

\begin{proposition}\label{prop:decreasing}
Given $f\colon X\to X$, $\ph\colon X\to \RR$, and $Z\subset X$, suppose there exist $\alpha, \beta \in \RR$ such that
\[
\alpha \leq \llim_{n\to\infty} \frac 1n S_n\ph(x) \leq \ulim_{n\to\infty} \frac 1n S_n\ph(x) \leq \beta
\]
for every $x\in Z$, and write $\gamma(t) = P_Z(t \ph)$.  Then the graph of $\gamma$ lies between the lines of slope $\alpha$ and $\beta$ through any point $(t,\gamma(t))\subset \RR^2$; that is,
\begin{equation}\label{eqn:cone}
\gamma(t) + \alpha h \leq \gamma(t+h) \leq \gamma(t) + \beta h
\end{equation}
for all $t\in \RR$, $h>0$.
\end{proposition}
\begin{proof}
Let $\eps>0$ be arbitrary.  Given $m\geq 1$, let
\[
Z_m = \left\{x\in Z \,\Big|\, \frac 1n S_n\ph(x)\in (\alpha-\eps,\beta+\eps) \text{ for all } n\geq m \right\},
\]
and observe that $Z=\bigcup_{m=1}^\infty Z_m$.  Now fix $t\in \RR$, $h>0$, and $N\geq m$.  It follows from the definition of $Z_m$ that for any $\delta>0$ and $s \in \RR$ we have
\begin{align*}
m_P(Z_m, s, &(t+h)\ph, N, \delta) \\ 
&=\inf_{\PPP(Z_m,N,\delta)} \sum_{(x_i,n_i)} \exp(-n_i s + (t+h)S_{n_i} \ph (x_i)) \\
&\geq 
\inf_{\PPP(Z_m,N,\delta)} \sum_{(x_i,n_i)} \exp(-n_i s + tS_{n_i} \ph (x_i) + n_i h(\alpha - \eps)) \\
&= m_P(Z_m, s - h(\alpha-\eps), t\ph, N, \delta).
\end{align*}
Letting $N\to\infty$, this gives
\[
m_P(Z_m, s, (t+h)\ph, \delta) \geq m_P(Z_m, s -h(\alpha-\eps), t\ph, \delta);
\]
in particular, if the second quantity is equal to $\infty$, then the first is as well.  Letting $\delta\to 0$, it follows that
\[
P_{Z_m}((t+h)\ph) \geq P_{Z_m}(t\ph) + h(\alpha - \eps).
\]
Taking the supremum over all $m\geq 1$ and using the fact that topological pressure is countably stable---that is, that $P_Z(\ph) = \sup_m P_{Z_m}(\ph)$ (see \cite[Theorem 11.2(3)]{yP98})---we obtain
\[
\gamma(t+h) \geq \gamma(t) + h(\alpha-\eps);
\]
since $\eps>0$ was arbitrary, this establishes the first half of~\eqref{eqn:decreasing}.  The second half is proved similarly; an analogous computation shows that
\[
m_P(Z_m, s, (t+h)\ph, N, \delta) \leq  m_P(Z_m, s - h(\beta+\eps), t\ph, N, \delta),
\]
whence upon passing to the limits and taking the supremum, we have
\[
\gamma(t+h) \leq \gamma(t) + h\beta.\qedhere
\]
\end{proof}

\begin{corollary}\label{cor:uniqueroot}
Let $f\colon X\to X$ be as in Theorem~\ref{thm:main}.  Fix $0 < \alpha \leq \beta < \infty$ and $Z\subset \AAA([\alpha,\beta])$, and write $\gamma(t) = P_Z(-t\log a)$.  Then the following hold:
\begin{enumerate}
\item  $\gamma$ is Lipschitz continuous with Lipschitz constant $\beta$ and strictly decreasing with rate at least $\alpha$; that is, for every $t\in \RR$ and $h>0$ we have
\begin{equation}\label{eqn:decreasing}
\gamma(t) - \beta h \leq \gamma(t+h) \leq \gamma(t) - \alpha h.
\end{equation}
\item The equation~\eqref{eqn:Bowen2} has a unique root $t^*$; furthermore,
\[
\frac{\htop(Z)}{\beta} \leq t^* \leq \frac{\htop(Z)}{\alpha}.
\]
\item If $\alpha=\beta$, so that $Z\subset \AAA(\alpha)$, then the unique root of~\eqref{eqn:Bowen2} is $t^*=\htop(Z)/\alpha$.
\end{enumerate}
\end{corollary}
\begin{proof}
(1) follows from Proposition~\ref{prop:decreasing} with $\ph=-\log a$.  (2) follows from the Intermediate Value Theorem by observing that the map $\tau\mapsto P_Z(-\tau\log a)$ is continuous and strictly decreasing, and that by~\eqref{eqn:decreasing} applied with $t=0$ and $h=\tau$, we have in the first place,
\[
P_Z(-\tau\log a) \geq P_Z(0) - \tau\beta = \htop(Z) - \tau\beta,
\]
so that $P_Z(-(\htop(Z)/\beta)\log a) \geq 0$, and in the second place,
\[
P_Z(-\tau\log a) \leq P_Z(0) - \tau\alpha = \htop(Z) - \tau\alpha,
\]
so that $P_Z(-(\htop(Z)/\alpha)\log a) \leq 0$.  Then (3) follows immediately.
\end{proof}

\begin{proposition}\label{prop:Lyap-tempered}
Let $f\colon X\to X$ be as in Theorem~\ref{thm:main}, and suppose that $\lambda(x)$ exists and is positive.  Then $x\in \BBB$.
\end{proposition}
\begin{proof}
Fix $\eps>0$ such that $\lambda(x) > \eps$, and choose $m\in\NN$ such that $|\lambda_n(x) - \lambda| < \eps$ for all $n\geq m$.  Let $\eta>0$ be such that
\[
\log \eta = \min_{0\leq j\leq m} \{S_j(\log a)(x) \} -\max_{0\leq k\leq m} \{ S_k(\log a)(x) \};
\]
thus for every $n\leq m$ and $0\leq k\leq n$, we have
\[
S_{n-k}\log a(f^k(x)) \geq \log \eta.
\]
Furthermore, for all $n \geq m$ and $0\leq k\leq n$, we have
\[
S_n(\log a)(x) \geq n(\lambda(x) - \eps),
\]
and either $0\leq k\leq m$, in which case $S_k(\log a) \leq -\log \eta$, or $m\leq k\leq n$, in which case
\[
S_k(\log a)(x) \leq k(\lambda(x) + \eps).
\]
Both these upper bounds are non-negative, and so together they imply
\[
S_k(\log a)(x) \leq -\log \eta + k(\lambda(x) + \eps),
\]
which yields
\begin{align*}
S_{n-k} (\log a)(f^k(x)) &= S_n\log a(x) - S_k\log a(x) \\
&\geq n(\lambda(x) - \eps) + \log \eta - k(\lambda(x) + \eps) \\
&\geq \log \eta - 2n\eps.
\end{align*}
It follows that $S_{n-k}(\log a)(f^k(x)) + 2n\eps \geq \log \eta$ for every $0\leq k\leq n$, and since $\eps>0$ was arbitrary, we have that $x$ satisfies~\eqref{eqn:tempered}.
\end{proof}

\section{Proof of Theorem~\ref{thm:main}}

In order to draw a connection between the Hausdorff dimension of $Z$ and the topological pressure of $-t\log a$ on $Z$, we need to establish a relationship between the two collections of covers $\DDD(Z,\eps)$ and $\PPP(Z,N,\delta)$.  Thus we prove the following lemma, which relates regular balls $B(x,r)$ to Bowen balls $B(x,n,\delta)$.

\begin{lemma}\label{lem:well-behaved}
Let $f\colon X\to X$ be as in Theorem~\ref{thm:main}.  Then given any $x\in \BBB$ and $\eps>0$, there exists $\delta_0=\delta_0(\eps)>0$ and $\eta=\eta(x,\eps)>0$ such that for every $n\in \NN$ and $0<\delta<\delta_0$,
\begin{equation}\label{eqn:diamball}
B\left(x,\eta\delta e^{-n(\lambda_n(x) + \eps)}\right) \subset B(x,n,\delta) \subset 
B\left(x,\delta e^{-n(\lambda_n(x) -\eps)}\right).
\end{equation}
\end{lemma}
\begin{proof}
Since $f$ is conformal with factor $a(x)>0$, we have
\[
\lim_{y\to x} \frac{d(f(x),f(y))}{d(x,y)} = a(x).
\]
Since $a(x) > 0$ everywhere, we may take logarithms and obtain
\[
\lim_{y\to x} \left(\log d(f(x),f(y)) - \log d(x,y)\right) = \log a(x).
\]
The pre-limit expression is a function on the direct product $X\times X$ with the diagonal $D=\{(x,x)\in X\times X\}$ removed; because $f$ is conformal, this function extends continuously to all of $X\times X$.  That is, there exists a continuous function $\zeta \colon X\times X\to \RR$ such that
\[
\zeta(x,y) = \begin{cases}
\log d(f(x),f(y)) - \log d(x,y) & x\neq y, \\
\log a(x) & x=y.
\end{cases}
\]
Because $X\times X$ is compact, $\zeta$ is uniformly continuous, hence given $\eps>0$ there exists $\delta_0=\delta_0(\eps)>0$ such that for every $0<\delta<\delta_0$ and $(x,y), (x',y')\in X\times X$ with
\[
(d\times d)((x,y),(x',y')) = d(x,x') + d(y,y') < \delta,
\]
we have $|\zeta(x,y) - \zeta(x',y')| < \eps$.  In particular, for $x,y\in X$ with $d(x,y)<\delta$, we have $(d\times d)((x,y),(x,x))<\delta$, and hence
\[
|\log d(f(x),f(y)) - \log d(x,y) - \log a(x)| = |\zeta(x,y) - \zeta(x,x)| < \eps.
\]
We may rewrite this inequality as
\[
\log d(f(x),f(y)) - \log a(x) - \eps < \log d(x,y) < \log d(f(x),f(y)) - \log a(x) + \eps,
\]
and taking exponentials, we obtain
\begin{equation}\label{eqn:danddf}
d(f(x),f(y)) e^{-(\log a(x) + \eps)} < d(x,y) < d(f(x),f(y)) e^{-(\log a(x) - \eps)}
\end{equation}
whenever the middle quantity is less than $\delta$.

We now show the second half of~\eqref{eqn:diamball}, and then go back and prove the first half.  Suppose $y\in B(x,n,\delta)$; that is, $d(f^k(y),f^k(x)) < \delta$ for all $0\leq k\leq n$.  Then repeated application of the second inequality in~\eqref{eqn:danddf} yields
\begin{align*}
d(x,y) &< d(f(x),f(y)) e^{-(\log a(x) - \eps)} \\
&< d(f^2(x),f^2(y)) e^{-(\log a(f(x)) - \eps)} e^{-(\log a(x) - \eps)} \\
&= d(f^2(x),f^2(y)) e^{- S_2 (\log a)(x) - 2\eps} \\
&< \cdots \\
&< d(f^n(x),f^n(y)) e^{-S_n (\log a)(x) - n\eps} \\
&< \delta e^{-n(\lambda_n(x) - \eps)}.
\end{align*}
The second inclusion in~\eqref{eqn:diamball} follows.

To prove the first inclusion in~\eqref{eqn:diamball}, we first suppose that $x$ has tempered contraction, and observe that if $d(x,y)<\delta$, then the first inequality in~\eqref{eqn:danddf} yields
\[
d(f(x),f(y)) < d(x,y) e^{\log a(x) + \eps}.
\]
Then if $d(x,y) < \delta e^{-(\log a(x)+\eps)}$, we have $d(f(x),f(y))<\delta$, and so
\begin{align*}
d(f^2(x),f(y)) &< d(f(x),f(y)) e^{\log a(f(x)) + \eps} \\
&< d(x,y) e^{2(\lambda_2(x) + \eps)}.
\end{align*}
Continuing in this manner, we see that if
\[
d(x,y) < \delta e^{-k(\lambda_k(x) + \eps)}
\]
for every $0\leq k\leq n$, we have $d(f^k(x),f^k(y))<\delta$ for every $0\leq k\leq n$, and hence $y\in B(x,n,\delta)$.  Thus we have proved that
\begin{equation}\label{eqn:almost}
B\left( x, \delta \min_{0\leq k\leq n} e^{-k(\lambda_k(x) + \eps)} \right) \subset B(x,n,\delta),
\end{equation}
which is almost what we wanted.  If the minimum was always achieved at $k=n$, we would be done; however, this may not be the case.  Indeed, if $\log a(f^n(x))<-\eps$ for some $n \in \NN$, then the minimum will be achieved for some smaller value of $k$.  

We now show that the tempered contraction assumption~\eqref{eqn:tempered} allows us to replace $e^{-k(\lambda_k(x) + \eps)}$ with the corresponding expression for $k=n$, at the cost of multiplying by some constant $\eta>0$ and replacing $\eps$ with $2\eps$.  To see what $\eta$ should be, we observe that
\begin{multline*}
\frac{e^{-n(\lambda_n(x) + 2\eps)}}{e^{-k(\lambda_k(x) + \eps)}}
= \frac{e^{-S_n\log a(x) - 2n\eps}}{e^{-{S_k\log a(x) - k\eps}}} \\
= e^{-(S_{n-k}\log a(f^k(x)) + 2n \eps - k\eps)} 
\leq e^{-(S_{n-k}\log a(f^k(x)) + n\eps)}.
\end{multline*}
Since $x$ has tempered contraction, there exists $\eta = \eta(x,\eps)>0$ such that
\begin{equation}\label{eqn:tempered2}
\log \eta < S_{n-k}(\log a)(f^k(x)) + n\eps
\end{equation}
for all $n\in \NN$, $0\leq k\leq n$, and hence
\[
e^{-(S_{n-k}(\log a)(f^k(x))+n\eps)} < 1/\eta.
\]
Thus for every such $n,k$, we have
\[
\eta e^{-n(\lambda_n(x) + 2\eps)} \leq e^{-k(\lambda_k(x) + \eps)},
\]
which along with~\eqref{eqn:almost} shows that
\[
B(x,\delta\eta e^{-n(\lambda_n(x)+2\eps)}) \subset B(x,n,\delta).
\]
Taking $\delta_0=\delta_0(\eps/2)$ gives the stated version of the result.  We remark that if $x$ has \emph{bounded} contraction, then $\eta = \eta(x)$ may be chosen independently of $\eps$.  Furthermore, if $a(x)\geq 1$ for all $x\in X$, then $\eta=1$ suffices.
\end{proof}

Using Lemma~\ref{lem:well-behaved}, we can prove the theorem for sets $Z\subset\AAA((\alpha,\infty))$, where $0<\alpha<\infty$; the general result will then follow from countable stability of topological pressure.  (Note that writing $\beta = \sup_{x\in X} \log a(x)$, we have $\AAA((\alpha,\infty)) \subset \AAA((\alpha,\beta])$; we do not allow maps with singularities, so all Lyapunov exponents are finite.)

\begin{lemma}\label{lem:dimheqt}
Let $f$ satisfy the conditions of Theorem~\ref{thm:main}, and fix a set $Z\subset \AAA((\alpha,\infty)) \cap \BBB$, where $0<\alpha<\infty$.  Let $t^*$ be the unique real number such that $P_Z(-t^*\log a)=0$, whose existence and uniqueness is guaranteed by Corollary~\ref{cor:uniqueroot}.  Then $\dim_H Z = t^*$.
\end{lemma}
\begin{proof}
First we show that $\dim_H Z\leq t^*$.  Given $m\geq 1$, consider the set
\[
Z_m = \{ x\in Z \mid \lambda_n(x) > \alpha \text{ for all $n\geq m$} \},
\]
and observe that $Z=\bigcup_{m=1}^\infty Z_m$.  Fix $t>t^*$; since $P_Z(-t\log a)<0$, there exists $\eps\in (0,\alpha)$ such that $-t\eps > P_Z(-t\log a)$.  By Lemma~\ref{lem:well-behaved}, there exists $\delta_0=\delta_0(\eps)>0$ such that for every $x\in Z_m$, $0<\delta\leq \delta_0$, and $n\geq m$, we have
\begin{equation}\label{eqn:diamball2}
\diam B(x,n,\delta) \leq 2\delta e^{-n(\lambda_n(x)-\eps)}
\leq 2\delta e^{-n(\alpha-\eps)} 
\end{equation}
Thus given $N> m$ and $0<\delta\leq \delta_0$, we have
\[
\PPP(Z_m,N,\delta) \subset \DDD\left(Z_m,2\delta e^{-N(\alpha -\eps)}\right).
\]
For any such $N$ and $\delta$, this allows us to relate the set functions which appear in the definitions of Hausdorff dimension and topological pressure as follows:
\begin{align*}
m_P(Z_m, -t\eps, &-t\log a, N, \delta) \\ 
&=\inf_{\PPP(Z_m,N,\delta)} \sum_{(x_i,n_i)} \exp(-n_i (-t\eps) - t S_{n_i}(\log a) (x_i)) \\
&=\inf_{\PPP(Z_m,N,\delta)} \sum_{(x_i,n_i)} \exp(-n_i t (\lambda_{n_i} (x_i) - \eps)) \\
&\geq \inf_{\PPP(Z_m,N,\delta)} \sum_{(x_i,n_i)} \left(\frac 1{2\delta} \diam B(x_i,n_i,\delta)\right)^t \\
&\geq \inf_{\DDD(Z_m,2\delta e^{-N(\alpha -\eps)})} \sum_{U_i} (2\delta)^{-t} (\diam U_i)^t \\
&= (2\delta)^{-t} m_H\left(Z_m,t, 2\delta e^{-N(\alpha -\eps)}\right)
\end{align*}
Taking the limit as $N\to\infty$ gives
\begin{equation}\label{eqn:mhleqmp}
m_P(Z_m, -t\eps, -t\log a, \delta) \geq (2\delta)^{-t} m_H(Z_m, t),
\end{equation}
for all $0<\delta<\delta_0$.  By our choice of $\eps$, we have
\[
-t\eps > P_Z (-t\log a) \geq P_{Z_m} (-t\log a) = \lim_{\delta\to 0} P_{Z_m} (-t\log a,\delta),
\]
and so for sufficiently small $\delta>0$, we have $-t\eps > P_{Z_m} (-t\log a,\delta)$, and hence $m_H(Z_m,t)=0$ by~\eqref{eqn:mhleqmp}, which implies $\dim_H(Z_m) \leq t$.

Since this holds for all $t>t^*$, we have $\dim_H(Z_m) \leq t^*$, and taking the union over all $m$ gives $\dim_H(Z) \leq t^*$.

For the other inequality, $\dim_H Z \geq t^*$, we fix $t<t^*$ and show that $\dim_H Z\geq t$.  We may assume that $t>0$, or there is nothing to prove.  By Corollary~\ref{cor:uniqueroot}, $t^*$ is the unique real number such that $P_Z(-t^*\log a)=0$, and since the pressure function is decreasing, we have $P_Z(-t\log a)>0$.  Thus we can choose $\eps>0$ such that
\[
0<t\eps <P_Z(-t\log a).
\]
Let $\delta_0 = \delta_0(\eps)$ be as in Lemma~\ref{lem:well-behaved}.  Given $m\geq 1$, consider the set
\[
Z_m = \{x\in Z \mid \text{\eqref{eqn:diamball} holds with $\eta=e^{-m}$ for all $n\in \NN$ and $0<\delta < \delta_0$} \}.
\]
Observe that $Z=\bigcup_{m=1}^\infty Z_m$, and so $P_Z(-t\log a) = \sup_m P_{Z_m} (-t\log a)$, where we once again use countable stability~\cite[Theorem 11.2(3)]{yP98}.  Thus there exists $m\in\NN$ such that $t\eps < P_{Z_m}(-t\log a)$, and we fix $0<\delta<\delta_0$ such that
\begin{equation}\label{eqn:epssmall}
t\eps < P_{Z_m}(-t\log a,\delta).
\end{equation}

Let $\beta = \sup_{x\in X} \log a(x) < \infty$.  Write $s_n(x) = e^{-m}\delta e^{-n(\lambda_n(x) + \eps)}$, and note that
\begin{equation}\label{eqn:notfast}
\frac{s_n(x)}{s_{n+1}(x)} = \frac{e^{-S_n \log a(x) - n\eps}}{e^{-S_{n+1} \log a(x) - (n+1)\eps}} = a(f^n(x)) e^\eps \leq e^{\beta + \eps}
\end{equation}
for every $n$ and $x$.  Furthermore, given $x\in Z_m$ and $r >0$ small, there exists $n=n(x,r)$ such that
\begin{equation}\label{eqn:diamleq}
s_n(x)e^{-(\beta+\eps)} \leq s_{n+1}(x) \leq r \leq s_n(x) = e^{-m} \delta e^{-n(\lambda_n(x) + \eps)}.
\end{equation}
For this value of $n$, Lemma~\ref{lem:well-behaved} implies that
\[
B(x,r) \subset B(x,n,\delta);
\]
consequently, given any $\{(x_i,r_i)\}$ such that $Z_m \subset \bigcup_i B(x_i, r_i)$, we also have $Z_m \subset \bigcup_i B(x_i,n_i,\delta)$, where $n_i = n(x_i,r_i)$ satisfies~\eqref{eqn:diamleq}.  

Furthermore, we have $\lambda_n(x) \leq \beta$ for all $n\in \NN$ and $x\in X$, and so $s_n(x) \geq \delta e^{-(m + n(\beta + \eps))}$.  It follows from~\eqref{eqn:diamleq} that for $n=n(x,r)$, we have
\[
\delta e^{-(m + (n+1)(\beta + \eps))} \leq r,
\]
and hence
\[
n \geq \frac{-\log r + \log \delta - m}{\beta+\eps} - 1.
\]
Denote the quantity on the right by $N(r,\delta)$, and observe that for each fixed $\delta>0$, we have $\lim_{r\to 0} N(r,\delta) = \infty$.  We see that the map $\{(x_i,r_i)\} \mapsto \{(x_i, n_i)\}$ defined above is a map from $\DDD^b(Z_m, r)$ to $\PPP(Z_m,N(r,\delta),\delta)$; thus~\eqref{eqn:diamleq} allows us to make the following computation for all $r>0$ and $0<\delta <\delta_0$:
\begin{align*}
m_H^{b'}(Z_m, t, r) 
&= \inf_{\DDD^b(Z_m,r)} \sum_{(x_i,r_i)} (2r_i)^t \\
&\geq \inf_{\PPP(Z_m, N(r,\delta), \delta)} \sum_{(x_i,n_i)} (2e^{-(\beta+\eps)} s_{n_i}(x))^t \\
&= (2\delta)^t e^{-t(m+\beta+\eps)} \inf_{\PPP(Z_m, N(r,\delta), \delta)} \sum_{(x_i,n_i)} e^{-n_i t(\lambda_{n_i}(x) + \eps)} \\
&= (2\delta)^t e^{-t(m+\beta+\eps)} m_P(Z_m, t\eps, -t\log a, N, \delta).
\end{align*}
Taking the limit as $r\to 0$ (and hence $N(r,\delta)\to\infty$), it follows from~\eqref{eqn:epssmall} that the quantity on the right goes to $\infty$, and so we have $m_H^{b'}(Z_m,t) = \infty$.  Using Proposition~\ref{prop:dimsequal}, this yields
\[
\dim_H Z \geq \dim_H Z_m \geq t,
\]
and since $t<t^*$ was arbitrary, this establishes the lemma.
\end{proof}

\begin{proof}[Proof of Theorem~\ref{thm:main}]
Fix a decreasing sequence of positive numbers $\alpha_k$ converging to $0$, and let $Z_k = Z\cap \AAA((\alpha_k,\infty))$, so that Lemma~\ref{lem:dimheqt} applies to $Z_k$, and we have $Z=\bigcup_{k=1}^\infty Z_k$.  For each $k$, let $t_k$ be the unique real number such that
\[
P_{Z_k}(-t_k\log a)=0;
\]
existence and uniqueness of $t_k$ are given by Corollary~\ref{cor:uniqueroot}.  Then Lemma~\ref{lem:dimheqt} shows that
\[
\dim_H Z_k = t_k.
\]
Writing $t^*= \sup_k t_k$, it follows that $\dim_H Z= t^*$, and it remains to show that
\begin{equation}\label{eqn:tstar}
t^* = \sup \{ t\geq 0 \mid P_Z(-t\log a)>0 \}.
\end{equation}
But given $t\geq 0$, we have
\[
P_Z(-t\log a) = \sup_k P_{Z_k}(-t\log a),
\]
and this is positive if and only if there exists $k$ such that $P_{Z_k}(-t\log a)>0$; that is, if and only if $t< t_k$.  This establishes~\eqref{eqn:tstar}.

Finally, it follows from~\eqref{eqn:tstar} and continuity of $t\mapsto P_Z(-t\log a)$ that $P_Z(-t^*\log a)=0$.  If $Z\subset \AAA((\alpha,\infty))$ for some $\alpha>0$, then Corollary~\ref{cor:uniqueroot} guarantees that $t^*$ is in fact the unique root of Bowen's equation.
\end{proof}

\emph{Acknowledgements.}  I would like to thank my advisor, Yakov Pesin, for the initial suggestion to consider the problems in multifractal formalism from which this result grew, and for tireless guidance and encouragement.  I would also like to thank: Yongluo Cao and the referee for finding several errors in the initial proofs; Dan Thompson for a conversation that revealed some technical issues with the definitions;  and the referee for making several suggestions that improved the exposition.

\def\cprime{$'$}
\providecommand{\bysame}{\leavevmode\hbox to3em{\hrulefill}\thinspace}
\providecommand{\MR}{\relax\ifhmode\unskip\space\fi MR }
\providecommand{\MRhref}[2]{%
  \href{http://www.ams.org/mathscinet-getitem?mr=#1}{#2}
}
\providecommand{\href}[2]{#2}


\def\cprime{$'$}
\begin{thebibliography}{GPR09}

\bibitem[Bow79]{rB79}
Rufus Bowen.
\newblock Hausdorff dimension of quasicircles.
\newblock {\em Inst. Hautes \'Etudes Sci. Publ. Math.}, (50):11--25, 1979.

\bibitem[BPS97]{BPS97}
Luis Barreira, Yakov Pesin, and J{\"o}rg Schmeling.
\newblock On a general concept of multifractality: Multifractal spectra for
  dimensions, entropies, and lyapunov exponents. multifractal rigidity.
\newblock {\em Chaos}, 7(1):27--38, 1997.

\bibitem[BS00]{BS00}
Luis Barreira and J{\"o}rg Schmeling.
\newblock Sets of ``non-typical'' points have full topological entropy and full
  {H}ausdorff dimension.
\newblock {\em Israel J. Math.}, 116:29--70, 2000.

\bibitem[Cli09]{vC09}
Vaughn Climenhaga.
\newblock Multifractal formalism derived from thermodynamics.
\newblock Preprint, 2009.

\bibitem[GP97]{GP97}
Dimitrios Gatzouras and Yuval Peres.
\newblock Invariant measures of full dimension for some expanding maps.
\newblock {\em Ergodic Theory Dynam. Systems}, 17(1):147--167, 1997.

\bibitem[GPR09]{GPR09}
Katrin Gelfert, Feliks Przytycki, and Micha{\l} Rams.
\newblock Lyapunov spectrum for rational maps.
\newblock 2009.
\newblock Preprint.

\bibitem[GR09]{GR09}
Katrin Gelfert and Micha{\l} Rams.
\newblock The {L}yapunov spectrum of some parabolic systems.
\newblock {\em Ergodic Theory Dynam. Systems}, 29:919--940, 2009.

\bibitem[Hu08]{hH08}
Huyi Hu.
\newblock Equilibriums of some non-{H}\"older potentials.
\newblock {\em Trans. Amer. Math. Soc.}, 360(4):2153--2190, 2008.

\bibitem[IK09]{IK09}
Godofredo Iommi and Jan Kiwi.
\newblock The {L}yapunov spectrum is not always concave.
\newblock {\em Journal of Statistical Physics}, 135(3):535--546, 2009.

\bibitem[MS00]{MS00}
N.~Makarov and S.~Smirnov.
\newblock On ``thermodynamics'' of rational maps {I}. {N}egative spectrum.
\newblock {\em Comm. Math. Phys.}, 211:705--743, 2000.

\bibitem[Nak00]{kN00}
Kentaro Nakaishi.
\newblock Multifractal formalism for some parabolic maps.
\newblock {\em Ergodic Theory Dynam. Systems}, 20(3):843--857, 2000.

\bibitem[Pes98]{yP98}
Yakov Pesin.
\newblock {\em Dimension Theory in Dynamical Systems: Contemporary Views and
  Applications}.
\newblock University of Chicago Press, 1998.

\bibitem[PP84]{PP84}
Ya.~B. Pesin and B.~S. Pitskel{\cprime}.
\newblock Topological pressure and the variational principle for noncompact
  sets.
\newblock {\em Funktsional. Anal. i Prilozhen.}, 18(4):50--63, 96, 1984.

\bibitem[PW97]{PW97}
Yakov Pesin and Howie Weiss.
\newblock The multifractal analysis of {G}ibbs measures: Motivation,
  mathematical foundation, and examples.
\newblock {\em Chaos}, 7(1):89--106, 1997.

\bibitem[PW99]{PW99}
Mark Pollicott and Howard Weiss.
\newblock Multifractal analysis of {L}yapunov exponent for continued fraction
  and {M}anneville-{P}omeau transformations and applications to {D}iophantine
  approximation.
\newblock {\em Comm. Math. Phys.}, 207(1):145--171, 1999.

\bibitem[Rue82]{dR82}
David Ruelle.
\newblock Repellers for real analytic maps.
\newblock {\em Ergodic Theory Dynamical Systems}, 2(1):99--107, 1982.

\bibitem[Wei99]{hW99}
Howard Weiss.
\newblock The {L}yapunov spectrum for conformal expanding maps and axiom-{A}
  surface diffeomorphisms.
\newblock {\em J. Statist. Phys.}, 95(3-4):615--632, 1999.

\end{thebibliography}


\begin{thebibliography}{10}

\bibitem{BPS97}
Luis Barreira, Yakov Pesin, and J{\"o}rg Schmeling.  On a general concept
  of multifractality: Multifractal spectra for dimensions, entropies, and
  lyapunov exponents. multifractal rigidity.  \emph{Chaos} \textbf{7}(1) (1997),
  27--38.

\bibitem{BS00} 
Luis Barreira and J{\"o}rg Schmeling.  Sets of ``non-typical'' points have
  full topological entropy and full {H}ausdorff dimension.  \emph{Israel J. Math.}
  \textbf{116} (2000), 29--70. 

\bibitem{rB73}
Rufus Bowen.  Topological entropy for noncompact sets.  \emph{Trans. Amer. Math.
  Soc.} \textbf{184} (1973), 125--136. 

\bibitem{rB79}
Rufus Bowen.  Hausdorff dimension of quasicircles.  \emph{Inst. Hautes \'Etudes Sci.
  Publ. Math.} \textbf{50} (1979), 11--25. 

\bibitem{vC10}
Vaughn Climenhaga.  Multifractal formalism derived from thermodynamics.
  \emph{Preprint} (2010).

\bibitem{DU91}
Manfred Denker and Mariusz Urba\'nski.  Ergodic theory of equilibrium
  states for rational maps.  \emph{Nonlinearity} \textbf{4} (1991), 103--134.

\bibitem{GP97}
Dimitrios Gatzouras and Yuval Peres.  Invariant measures of full dimension
  for some expanding maps.  \emph{Ergodic Theory Dynam. Systems} \textbf{17}(1) (1997),
  147--167.

\bibitem{GPR09}
Katrin Gelfert, Feliks Przytycki, and Micha{\l} Rams.  Lyapunov spectrum
  for rational maps.  \emph{Preprint} (2009).

\bibitem{GR09}
Katrin Gelfert and Micha{\l} Rams.  The {L}yapunov spectrum of some
  parabolic systems.  \emph{Ergodic Theory Dynam. Systems} \textbf{29} (2009),
  919--940.

\bibitem{hH08}
Huyi Hu.  Equilibriums of some non-{H}\"older potentials.  \emph{Trans. Amer.
  Math. Soc.} \textbf{360}(4) (2008), 2153--2190.

\bibitem{IK09}
Godofredo Iommi and Jan Kiwi.  The {L}yapunov spectrum is not always
  concave. \emph{Journal of Statistical Physics} \textbf{135}(3) (2009),
  535--546.

\bibitem{MS00}
N.~Makarov and S.~Smirnov.  On ``thermodynamics'' of rational maps {I}.
  {N}egative spectrum.  \emph{Comm. Math. Phys.} \textbf{211} (2000), 705--743.

\bibitem{MU08}
Volker Mayer and Mariusz Urba{\'n}ski.  Geometric thermodynamic formalism
  and real analyticity for meromorphic functions of finite order.  \emph{Ergodic
  Theory Dynam. Systems} \textbf{28}(3) (2008), 915--946. 

\bibitem{MU10}
Volker Mayer and Mariusz Urba{\'n}ski.  Thermodynamical formalism and multifractal 
  analysis for meromorphic functions of finite order.  \emph{Memoirs of the AMS} 
  \textbf{203}(954) (2010).

\bibitem{kN00}
Kentaro Nakaishi.  Multifractal formalism for some parabolic maps.  \emph{Ergodic 
  Theory Dynam. Systems} \textbf{20}(3) (2000), 843--857.

\bibitem{PP84}
Ya.~B. Pesin and B.~S. Pitskel{\cprime}.  Topological pressure and the
  variational principle for noncompact sets.  \emph{Funktsional. Anal. i Prilozhen.}
  \textbf{18}(4) (1984), 50--63, 96. 

\bibitem{yP98}
Yakov Pesin.  \emph{Dimension theory in dynamical systems: Contemporary views
  and applications}, University of Chicago Press, 1998.

\bibitem{PW97}
Yakov Pesin and Howie Weiss.  The multifractal analysis of {G}ibbs
  measures: Motivation, mathematical foundation, and examples.  \emph{Chaos}
  \textbf{7}(1) (1997), 89--106.

\bibitem{PW99}
Mark Pollicott and Howard Weiss.  Multifractal analysis of {L}yapunov
  exponent for continued fraction and {M}anneville-{P}omeau transformations and
  applications to {D}iophantine approximation.  \emph{Comm. Math. Phys.} \textbf{207}(1)
  (1999), 145--171. 

\bibitem{PRS03}
Feliks Przytycki, Juan Rivera-Letelier, and Stanislav Smirnov.
  Equivalence and topological invariance of conditions for non-uniform
  hyperbolicity in the iteration of rational maps.  \emph{Invent. Math.} \textbf{151}(1)
  (2003), 29--63. 

\bibitem{PRS04}
Feliks Przytycki, Juan Rivera-Letelier, and Stanislav Smirnov.  Equality of 
  pressures for rational functions.  \emph{Ergodic Theory
  Dynam. Systems} \textbf{24}(3) (2004), 891--914. 

\bibitem{dR82}
David Ruelle.  Repellers for real analytic maps, \emph{Ergodic Theory Dynamical
  Systems} \textbf{2}(1) (1982), 99--107.

\bibitem{hhR08}
Hans~Henrik Rugh.  On the dimensions of conformal repellers. {R}andomness
  and parameter dependency.  \emph{Ann. of Math. (2)} \textbf{168}(3) (2008),
  695--748. 

\bibitem{mU91}
M.~Urba{\'n}ski. On the {H}ausdorff dimension of a {J}ulia set with a
  rationally indifferent periodic point.  \emph{Studia Math.} \textbf{97}(3) (1991),
  167--188. 

\bibitem{mU96}
Mariusz Urba{\'n}ski.  Parabolic {C}antor sets.  \emph{Fund. Math.} \textbf{151}(3)
  (1996), 241--277. 

\bibitem{UZ04}
Mariusz Urba{\'n}ski and Anna Zdunik.  Real analyticity of {H}ausdorff
  dimension of finer {J}ulia sets of exponential family.  \emph{Ergodic Theory Dynam.
  Systems} \textbf{24}(1) (2004), 279--315. 

\bibitem{hW99}
Howard Weiss.  The {L}yapunov spectrum for conformal expanding maps and
  axiom-{A} surface diffeomorphisms.  \emph{J. Statist. Phys.} \textbf{95}(3--4) (1999),
  615--632. 

\end{thebibliography}
\end{document}